\documentclass{amsart}
\usepackage{latexsym,amssymb,amsmath,amsthm,amsfonts}
\usepackage{hyperref}
\usepackage[shortlabels]{enumitem}
\usepackage{color}

\newtheorem{theorem}{Theorem}[section]
\newtheorem{corollary}[theorem]{Corollary}
\newtheorem{lemma}[theorem]{Lemma}
\newtheorem{proposition}[theorem]{Proposition}

\newtheorem*{ThmA}{Theorem A}
\newtheorem*{ThmB}{Theorem B}

\theoremstyle{definition}
\newtheorem{example}[theorem]{Example}
\newtheorem{rem}[theorem]{Remark}

\newenvironment{enumeratei}{\begin{enumerate}[\upshape (a)]}
    {\end{enumerate}}

\def\irr#1{{\rm Irr}(#1)}

\def\cs#1{{\rm cs}(#1)}
\def\vcs#1{{\rm vcs}(#1)}

\def\cent#1#2{{\bf C}_{#1}(#2)}

\def\oh#1#2{{\bf O}_{#1}(#2)}

\def\zent#1{{\bf Z}(#1)}

\def\fit#1{{\bf F}(#1)}

\begin{document}

\title[One vanishing class size]{On solvable groups with one vanishing class size}

\author[M. Bianchi et al.]{M. Bianchi}
\address{Mariagrazia Bianchi, Dipartimento di Matematica F. Enriques,\newline
Universit\`a degli Studi di Milano, via Saldini 50,
20133 Milano, Italy.}
\email{mariagrazia.bianchi@unimi.it}

\author[]{R. D. Camina}
\address{Rachel D. Camina, Fitzwilliam College, \newline Cambridge, CB3 0DG,
UK.} \email{rdc26@dpmms.cam.ac.uk}

\author[]{Mark L. Lewis}
\address{Mark L. Lewis, Department of Mathematical Sciences, \newline Kent State University, Kent, Ohio 44242 USA.}
\email{lewis@math.kent.edu}

\author[]{E. Pacifici}
\address{Emanuele Pacifici, Dipartimento di Matematica F. Enriques,
\newline Universit\`a degli Studi di Milano, via Saldini 50,
20133 Milano, Italy.}
\email{emanuele.pacifici@unimi.it}

\dedicatory{Dedicated to Carlo Casolo}

\thanks{
\\
\indent This research was supported by the Italian PRIN 2015TW9LSR\_006 ``Group Theory and Applications". The first and fourth authors are partially supported by INdAM-GNSAGA. The second author is partially supported by the Isaac Newton Trust.}



\begin{abstract}
Let \(G\) be a finite group, and let \(\cs G\) be the set of conjugacy class sizes of \(G\). Recalling that an element \(g\) of \(G\) is called a \emph{vanishing element} if there exists an irreducible character of \(G\) taking the value \(0\) on \(g\), we consider one particular subset of \(\cs G\), namely, the set \(\vcs G\) whose elements are the conjugacy class sizes of the vanishing elements of \(G\). Motivated by the results in~\cite{BLP}, we describe the class of the finite groups \(G\) such that \(\vcs G\) consists of a single element \emph{under the assumption that \(G\) is supersolvable or \(G\) has a normal Sylow \(2\)-subgroup} (in particular, groups of odd order are covered). As a particular case, we also get a characterization of finite groups having a single vanishing conjugacy class size \emph{which is either a prime power or square-free}.
\end{abstract}

\maketitle

\section{Introduction}

Given a finite group \(G\), let \(\cs G\) denote the set consisting of the sizes of the conjugacy classes of~\(G\). 
A well-established research field in the theory of finite groups investigates the interplay between some features of this set of positive integers and the structure of the group itself (see \cite{cam} for an overview of this subject);  in particular, several authors investigated the so-called \emph{conjugate rank} of a group~\(G\), defined as the number of elements in \(\cs G\setminus\{1\}\), and its influence on the group structure. We mention, for instance, the work by N. It\^o (\cite{ito}) and by K. Ishikawa~(\cite{ish}) showing that a finite group with conjugate rank \(1\) is, up to abelian direct factors, a group of prime-power order of nilpotency class at most \(3\); also, the class of finite groups of conjugate rank \(2\) was studied by S.~Dolfi and E. Jabara in~\cite{dolfi}. 

Recently, this research area has been intertwined with character theory via the concept of \emph{vanishing elements}, introduced by I.M. Isaacs, G. Navarro and T.R. Wolf in \cite{wolf}: an element \(g\) of a finite group \(G\) is called a vanishing element of \(G\) if there exists an irreducible character of \(G\) taking the value \(0\) on \(g\), and the conjugacy class of such an element is called a vanishing conjugacy class of $G$. Now, one can focus on a subset of \(\cs G\) ``filtered" by means of the irreducible characters of \(G\), considering the set \(\vcs G\) of vanishing conjugacy class sizes of \(G\) (see also \cite{DPSex} for a survey on results concerning vanishing elements and vanishing conjugacy classes). In view of the previous paragraph, a natural issue in this context can be to investigate the structure of finite groups having a unique (non-central) vanishing conjugacy class size. This question was considered in \cite{BLP}, where a characterization is obtained under the assumption that the unique vanishing conjugacy class size is a prime number \(p\), namely: \(\vcs G=\{p\}\) if and only if either \(G\) is (up to abelian direct factors) a \(p\)-group with \(\cs G=\{1,p\}\), or \(G/\zent G\) is a Frobenius group whose Frobenius kernel has order \(p\).

One of the main results of this paper is a generalization of the aforementioned Theorem~{1} of \cite{BLP}.

\begin{ThmA}
Let $G$ be a finite group, let $s$ be an integer that is either a prime power or square-free. Then, denoting by $\pi$ the set of prime divisors of $s$, we have $\vcs G = \{ s \}$ if and only if $G = NH$, where $N$ is a normal Hall $\pi$-subgroup and $H$ an abelian $\pi$-complement of \(G\), and one of the following occurs.
\begin{enumeratei}
\item The integer \(s\) is a prime power, $G = N \times H$, and ${\rm cs} (N) = \{ 1, s \}$.
\item $G/\zent G$ is a Frobenius group whose Frobenius kernel $N\zent G/\zent G$ has order~$s$, a Frobenius complement is $H\zent G/\zent G$, and $N$ contains no vanishing elements of $G$.
\end{enumeratei}
\end{ThmA}

We note that (a) occurs if and only if \(\pi\) equals the set \(\pi(G/\zent G)\) of prime divisors of \(|G/\zent G|\), whereas (b) occurs if and only if \(\pi\) is strictly contained in that set.

As we see from the above result, the finite groups \(G\) such that \(\vcs G\) consists of a single element \(s\) which is either a prime power or square-free turn out to have a \emph{nilpotent} Hall \(\pi\)-subgroup (where \(\pi\) is the set of prime divisors of \(s\)). In fact, a crucial step in our proof of Theorem~A is to characterize the class of finite groups \(G\) satisfying the condition \(\vcs G=\{s\}\) and having a nilpotent Hall subgroup for the set of prime divisors of \(s\); this is achieved in Theorem~\ref{nilpotent}.

We also remark that there exist groups \(G\) as in conclusion (b) of Theorem~A for which \(N\zent G/\zent G\) is nonabelian, and also groups such that \(G/\zent G\) is a Frobenius group whose kernel \(N\zent G/\zent G\) has prime-power order but \(N\) does contain vanishing elements of \(G\); we will present some examples after the proof of Theorem~\ref{nilpotent}.

\smallskip
As another remark, the groups as in Theorem~A turn out to be nilpotent-by-abelian (a necessary condition for a group in order to be supersolvable) and, in the square-free case, they are in fact supersolvable. One may wonder if the supersolvable groups having a single vanishing class size are necessarily as in (a) and (b) of Theorem~A. 

Actually, if the supersolvable group \(G\) is such that \(\vcs G=\{s\}\), \emph{and the set of prime divisors of \(s\) is strictly smaller than \(\pi(G/\zent G)\)}, then \(G\) is as in conclusion (b) of Theorem~A, and the same holds also replacing the supersolvability assumption with \(G\) having a normal Sylow \(2\)-subgroup (see Theorem~\ref{supersolvable 1}). On the other hand, if $\pi=\pi (G/\zent G)$, then the situation is different, as shown by the following result.

\begin{ThmB}  
Let $G$ be a finite supersolvable group, and assume that $\vcs G = \{ s \}$.  If, denoting by \(\pi\) the set of prime divisors of \(s\), we have $\pi=\pi (G/\zent G)$, then, up to an abelian direct factor of \(\pi'\)-order, one of the following occurs.
\begin{enumeratei}
\item The integer \(s\) is power of a suitable prime \(p\), \(G\) is a \(p\)-group and $\cs G = \{ 1, s \}$.
\item For a suitable prime \(p\) we have $G = N P$, where $N$ is a nontrivial, abelian, normal $p$-complement of \(G\), and $P$ is a Sylow $p$-subgroup of \(G\) such that $|\cs P| = 2$, $\zent P = \cent P N$, and \(P/\zent P\) is an elementary abelian \(p\)-group. Also, $|N:\cent N x|$ has the same value for every element $x \in P \setminus \zent P$, \(\cent N P=N\cap\zent G\), and \(N\) contains no vanishing elements of \(G\).
\end{enumeratei}
\end{ThmB}

When we prove Theorem B in Section~5, we will obtain some additional technical conditions in~(b) that will enable us to have in fact an ``if and only if" theorem (see Theorem~\ref{thmB}).  As we shall then discuss in Section~6, the class of groups as in (b) of Theorem~B is non-empty. Also, similar conclusions can be obtained replacing the assumption of supersolvability of \(G\) with the assumption that \(G\) has a normal Sylow \(2\)-subgroup (see Theorem~\ref{pi(s)=pi(G)}); in that case a third class of (non-supersolvable) groups arises in principle, but at the time of this writing we are not aware of any example of a group as in that class. 

Another question that we leave as an open problem is whether a finite group having a single vanishing conjugacy class size is necessarily solvable. As a first step, we prove in Lemma~\ref{abelian} that such a group does not have any nonabelian minimal normal subgroup, thus, in particular, it is not an almost-simple group. At any rate, our conjecture is that the finite groups having a unique vanishing conjugacy class size are only those described in this paper; we refer the reader to Section~6 for some more details.

To close with, we mention that a key tool in our proofs of Theorem~A and Theorem~B is a result by I.M. Isaacs (see \cite{isaacs}), where the author describes finite groups having a normal subgroup such that all elements outside this subgroup lie in conjugacy classes of the same size; in fact, in Theorem~\ref{same size}, we provide a neat characterization of one particular situation of this kind, when the relevant group has a normal Hall subgroup. 

Throughout the following discussion, every group is assumed to be a finite group.

\section{Preliminary results and notation}

If \(n\) is a positive integer, we denote by \(\pi(n)\) the set of prime divisors of \(n\) and, for a group \(G\), we write \(\pi(G)\) for \(\pi(|G|)\). Also, given a set \(\pi\) of prime numbers, \(n_{\pi}\) and \(n_{\pi'}\) will denote the \(\pi\)- and the \(\pi'\)-part of \(n\), respectively. 

As mentioned in the Introduction, \(\vcs G\) is defined as the set of vanishing conjugacy class sizes of the group \(G\). We will freely use some elementary properties of conjugacy class sizes and of vanishing elements, as, for instance, the fact that if \(N\) is a normal subgroup of \(G\) and \(gN\) is a vanishing element of \(G/N\), then every element in the coset \(gN\) is a vanishing element of \(G\).


We start with some general lemmas.

\begin{lemma} \label{first}
Let $G$ be a group, and let $\pi$ be a set of primes. If every element of $\vcs G$ is divisible only by primes in $\pi$, then the following conclusions hold.
\begin{enumeratei} 
\item $G = NH$, where $N$ is a normal Hall $\pi$-subgroup of $G$ and $H$ is an abelian $\pi$-complement. 
\item \label{conjugates} If \(x\in N\) is a vanishing element of $G$, then \(x\) lies in $\cent N H^y$ for a suitable \(y\in N\).
\end{enumeratei}
\end{lemma}

\begin{proof}
By Theorem~A of \cite{DPS}, under our assumptions, \(G\) has a normal \(q\)-complement and abelian Sylow \(q\)-subgroups for every prime \(q\) lying in \(\pi(G)\setminus\pi\). Therefore \(G\) has a normal Hall \(\pi\)-subgroup \(N\), and a Hall \(\pi'\)-subgroup \(H\) of \(G\) is nilpotent because it has a normal \(q\)-complement for every prime \(q\) dividing its order; but \(H\) is in fact abelian as all its Sylow subgroups are abelian, and (a) is proved.

As for (b), assume that $x \in N$ is a vanishing element of $G$: by our hypothesis, we know that $|G:\cent G x|$ is divisible only by primes in \(\pi\), and therefore $\cent G x$ contains a conjugate \(H^g\) of the \(\pi\)-complement $H$. Since $H^g$ can be clearly written as \(H^y\) for a suitable element \(y\) of \(N\), we conclude that \(x\) lies in $\cent N H^y$, as claimed.  
\end{proof}



\begin{lemma}\label{hypothesis1}
Let \(G\) be a group, let \(\pi\) be a set of primes, and assume that \(G\) has a normal Hall \(\pi\)-subgroup \(N\). Then, denoting by \(H\) a \(\pi\)-complement of \(G\),  the following conclusions hold.
\begin{enumeratei}
\item We have $\cent N H \cap \zent N = \zent G \cap N$. Furthermore, if $\cent N H$ is abelian, then also the normal core of \(\cent N H\) in \(N\) is $\zent G \cap N$.
\item If \(H\) is abelian, then $\cent H N = \zent G \cap H$.  Furthermore, for \(g\in G\), the conjugacy class size of $g\cent H N$ in \(G/\cent H N\) is the same as the conjugacy class size of \(g\) in \(G\) and, in particular, $\vcs {G/\cent H N} \subseteq \vcs G$.
\item If $\nu \in \irr N$ vanishes on an element $x \in \cent N H$, then $x$ is a vanishing element of $G$.
\end{enumeratei}
\end{lemma}
\begin{proof}

We start by proving (a). Clearly we have $\zent G \cap N \le \cent N H \cap \zent N$.  On the other hand, $N$ centralizes $\zent N$ and \(H\) centralizes $\cent N H$, so $G=HN$ centralizes $\cent N H \cap \zent N$.  Since $\cent N H \cap \zent N$ is contained in $N$, we obtain $\cent N H \cap \zent N \le \zent G \cap N$, as desired. Observe also that, 
as $\zent G \cap N$ is a normal subgroup of \(N\) contained in $\cent N H$, we get $\zent G \cap N \le {\rm Core}_N (\cent N H)$.  

Now, by coprimality we have $N = [N,H] \cent N H$, whence $G = H[N,H]\cent N H$.  Taking into account that every conjugate of $H$ in \(G\) is in fact a conjugate of \(H\) by an element of \(N\), it follows that the normal closure $H^G$ of \(H\) in \(G\) is  $\langle H^y \mid y \in N \rangle = [H,N]H$.  Thus, $G = HN = H[H,N]\cent N H = H^G \cent N H$.  Since ${\rm Core}_N (\cent N H) = \bigcap_{y \in N} \cent N {H^y}$, we see that $H^y$ centralizes ${\rm Core}_N (\cent N H)$ for each $y \in N$, and therefore $H^G$ centralizes ${\rm Core}_N (\cent N H)$.  Assuming now that \(\cent N H\) is abelian, we clearly get that the whole \(G=H^G \cent N H\) centralizes ${\rm Core}_N (\cent N H)$, which is therefore contained in $\zent G \cap N$. 

\smallskip
We move next to part (b), so we assume that \(H\) is abelian. It is obvious that $\zent G \cap H \le \cent H N$.  On the other hand, $N$ centralizes $\cent H N$ and, since $H$ is abelian, $H$ centralizes $\cent H N$ as well.  Thus $\cent H N$ is central in $HN = G$ and we have $\cent H N = \zent G \cap H$.

As for the remaining claims of (b), since $\cent H N$ lies in \(\zent G\), for \(g\in G\) we have $\cent H N \le \cent G g$. Set $U/\cent H N$ to be the centralizer in $G/\cent H N$ of $g\cent H N$.  It is obvious that $\cent G g \le U$.  On the other hand we have that $[U,g]$ lies in $\cent H N \cap G'$, but $G'$ lies in $N$, so $[U,g]$ is in fact in $\cent H N \cap N=1$.  This implies $U \le \cent G g$, and so, $U = \cent G g$. 
As a consequence, we get \(|G/\cent H N:\cent{G/\cent H N}{g\cent H N}|=|G:\cent G g|\), and the proof of (b) can be easily concluded by observing that if \(g\cent H N\) is a vanishing element of \(G/\cent H N\), then \(g\) is a vanishing element of \(G\). 

\smallskip
Finally, we prove (c). Our assumption is $\nu (x) = 0$, where \(\nu\in\irr N\) and \(x\in\cent N H\). For \(g\) in \(G\), we can write $g^{-1} = hn$ for suitable elements $h \in H$ and $n \in N$, so we get $x^{g^{-1}} = x^{hn} = x^n$, because $x \in \cent N H$. Now, $\nu^g (x)=\nu(x^{g^{-1}}) = \nu (x^n)$ and, since $\nu$ is a class function of \(N\), we have $\nu^g (x) = 0$. As this holds for every element $g \in G$, it follows by Clifford Theory that any irreducible constituent of $\nu^G$ vanishes on~\(x\). 
\end{proof}

As we will see, a situation that turns out to be very relevant in our analysis is when a group $G$ has a normal subgroup $N$ such that all elements of $G\setminus N$ lie in conjugacy classes of the same size. This situation was studied in \cite{isaacs}. 

\begin{theorem}\cite{isaacs} \label{isaacs}
Let $G$ be a nonabelian group with a proper normal subgroup $N$ such that all of the conjugacy class of $G$ which lie outside of $N$ have equal sizes.  Then either $G/N$ is cyclic, or else every nonidentity element of $G/N$ has prime order. In the first situation, $G$ has an abelian Hall $\pi$-subgroup and a normal $\pi$-complement, where $\pi$ is the set of primes dividing the index $|G:N|$.
\end{theorem}

When we are in the situation of Lemma~\ref{hypothesis1}, we can characterize the groups that satisfy the hypotheses of Theorem~\ref{isaacs}.  Let $\pi$ be a set of primes.  Following the literature, a group $G$ is $\pi$-separable if $G$ has a composition series where the composition factors are all $\pi$-groups or $\pi'$-groups.  It is known that $\pi$-separable groups have Hall $\pi$-subgroups and that every $\pi$-subgroup is contained in some Hall $\pi$-subgroup.

\begin{theorem} \label{same size}
Let \(G\) be a group, let \(\pi\) be a set of primes, and assume that \(G\) has a normal Hall \(\pi\)-subgroup \(N\). Also, denote by \(H\) a \(\pi\)-complement of \(G\).  Then every element of $G$ outside of $N \cent H N$ lies in conjugacy classes of the same size $s$ if and only if the following hold.
\begin{enumeratei}
	\item $|H:\cent H h|$ has the same value for all $h \in H \setminus \cent H N$.
	\item $|N:\cent N h|$ has the same value for all $h \in H \setminus \cent H N$.
	\item $\cent N {\cent H h} = \cent N h$ for every element $h \in H \setminus \cent H N$.  
	\item $\cent N h$ is abelian for every element $h \in H \setminus \cent H N$.
\end{enumeratei} 
In this situation,  $|H:\cent H h|$ is $s_{\pi'}$ and $|N:\cent N h|$ is $s_\pi$  for all $h \in H \setminus \cent H N$.
\end{theorem}

\begin{proof}
We first assume that every element of $G$ outside of $N \cent H N$ lies in conjugacy classes of the same size.  Fix an element $h \in H \setminus \cent H N$.  We always have $\cent N{\cent H h} \le \cent N h$.  Suppose $1 \ne n \in \cent N h$.  Then $hn$ is not an element of $N \cent H N$.  Since $h$ is also not an element of $N\cent H N$, we have $|G:\cent G h| = |G:\cent G {hn}|$.  Since $n$ and $h$ commute and have coprime orders, we know that $\cent G {hn} = \cent G h \cap \cent G n$, and we obtain $\cent G {hn} \le \cent G h$.  We have shown that these subgroups have the same indices in $G$, so $\cent G {hn} = \cent G h$.  We now have $\cent H h \le \cent G h = \cent G {hn} \le \cent G n$, and in particular, $n \in \cent N{\cent H h}$.  This yields $\cent N{\cent H h} = \cent N h$, proving (c).    It follows that $\cent N h \le \cent G n$ for every element $n \in \cent N h$.  We observe that $\cent N h$ centralizes all of its elements, and hence, $\cent N h$ is abelian.  This proves (d).

We know that $|G:\cent G h|$ is the class size for elements outside of $N \cent H N$.  Since $N$ is a normal Hall $\pi$-subgroup of $G$, it follows that $\cent N h = \cent G h \cap N$ is the Hall $\pi$-subgroup of $\cent G h$.  This yields $|N:\cent N h| = |G:\cent G h|_{\pi}$, and thus, we have that $|N:\cent Nh|$ will be the same for all elements $h \in H \setminus \cent H N$.   This proves (b).  Notice that $G$ is $\pi$-separable, so there is an element $g \in G$ so that $H^g \cap \cent G h$ is a $\pi$-complement of $\cent G h$ containing $h$.  Since $G = HN$, we see that $H^g = H^m$ for some element $m \in N$.  We now have $h \in H \cap H^m$, and so, $h^{m^{-1}} \in H$.  We obtain $[h,m^{-1}] = h^{-1}h^{m^{-1}} \in H$.  Since $m^{-1} \in N$, we have $[h,m^{-1}] \in N$, and using the fact that $N \cap H = 1$, we deduce that $[h,m^{-1}] = 1$.  Thus, $h$ and $m$ commute, and so, $m \in \cent N h = \cent N{\cent H h}$.  It follows that $\cent {H^m} h = \cent H h$.  Observe that $|G:\cent G h|_{\pi'} = |H^m:\cent {H^m} h| = |H:\cent H h|$, and therefore, this value is the same for all $h \in H \setminus \cent H N$ which proves (a).

Conversely, we assume (a)--(d).  First, consider $h \in H \setminus \cent H N$.  Observe that $\cent N h$ is the Hall $\pi$-subgroup of $\cent G h$.  There exists a $\pi$-complement $H_1$ of $G$ so that $H_1 \cap \cent G h$ is a $\pi$-complement of $\cent G h$ and $h \in H_1$.  Since $H$ and $H_1$ are conjugate and elements in $H \setminus \cent H N$ have the same conjugacy class size in $H$, we have that $|H_1:\cent {H_1} h| = |H:\cent H h|$, and so, $\cent H h$ is a $\pi$-complement of $\cent G h$.  This implies that $\cent G h = \cent N h \cent H h$.  Furthermore, $|G:\cent G h| = |H:\cent H h| |N:\cent N h|$, and by hypothesis, this is the same value for all elements $h \in H \setminus \cent H N$.

Suppose $g \in G \setminus N \cent H N$.  We know that there exist unique elements $g_1,g_2 \in G$ so that $g = g_1 g_2$, $g_1$ has $\pi$-order, $g_2$ has $\pi'$-order, and $g_1$ and $g_2$ commute.  We know $g_1 \in N$ and $g_2$ is conjugate to an element of $H \setminus \cent H N$.  Conjugating, we may assume that $g_1 = n \in N$ and $g_2 = h \in H \setminus \cent H N$.  Since $n$ and $h$ commute and have coprime orders, we have $\cent G g = \cent G h \cap \cent G n$ and $n \in \cent N h$.  This implies that $\cent G g \le \cent G h$.  Observe that $n \in \cent N h = \cent N{\cent H h}$, and so, $\cent H h \le \cent G n$.  Since $\cent N h$ is abelian, and $n \in \cent N h$, we have $\cent N h \le \cent G n$.  It follows that $\cent G h = \cent H h \cent N h \le \cent G n$, hence $\cent G h = \cent G h \cap \cent G n = \cent G g$.  In light of the previous paragraph, we have that $|G:\cent G g|$ is the same for all elements $g \in G \setminus N \cent H N$.
\end{proof}

To close this preliminary section, under the additional assumptions of \(H\) being abelian and not centralizing \(N\) we simplify the above theorem further and derive some more conclusions.

\begin{corollary} \label {same size 1}
Let \(G\) be a group, let \(\pi\) be a set of primes, and assume that \(G\) has a normal Hall \(\pi\)-subgroup \(N\). Also, denoting by \(H\) a \(\pi\)-complement of \(G\), assume that \(H\) is abelian and that \(\cent H N< H\).  Then the following conclusions hold.
\begin{enumeratei}
\item All elements of $G \setminus N \cent H N$ have the same conjugacy class size if and only if $\cent N h = \cent N H$ for all $h \in H \setminus \cent H N$ and $\cent N H$ is abelian. 
\item If the (equivalent) conditions of {\rm{(a)}} hold and \(\cent N H\) is normal in \(N\), then $G/\zent G$ is a Frobenius group with Frobenius kernel $N\zent G/\zent G$.
\end{enumeratei}
\end{corollary}

\begin{proof}
Since $H$ is abelian, in particular we have that $\cent H h = H$ for all elements $h \in H \setminus \cent H N$.  Notice that (a) of Theorem \ref{same size} is trivially met, and \ref{same size}(c) becomes $\cent N h = \cent N H$, and so, \ref{same size}(b) becomes an immediate consequence of \ref{same size}(c).  Hence, in this case, all of the elements in $G \setminus N \cent N H$ have the same class size if and only if $\cent N h = \cent N H$ for every element $h \in H \setminus \cent H N$ and $\cent N H$ is abelian, thus our first claim is proved. 

As for the second claim, assume now that we are in the above situation, and that \(\cent N H\) is normal in \(N\). In light of Lemma \ref{hypothesis1}(a), we have $\cent N H = \zent G \cap N$.  Observe that if $h \in H \setminus \cent H N$, then by (a) we have $\cent {N/\cent N H} {h\cent H N} = \cent {N/\cent N H} {H/\cent H N} = 1$. Hence $H/\cent H N$ acts fixed-point freely on $N/\cent N H$, and the conclusion follows. 
\end{proof}

\section{When the unique vanishing class size is a prime power}

Let $G$ be a group such that \(\vcs G=\{s\}\), where \(s\) is a positive integer whose set of prime divisors we denote by \(\pi\). By Lemma~\ref{first}(a) we know that $G = NH$, where $N$ is a normal Hall $\pi$-subgroup of $G$ and $H$ is an abelian $\pi$-complement. The main result of this section, which is Theorem~\ref{nilpotent}, yields a characterization of a group \(G\) of this kind under the assumption that \(N\) is nilpotent; as an immediate consequence, in Corollary~\ref{prime power} we will prove the part of Theorem~A regarding the case when \(s\) is a prime power.

The following lemma will be a useful tool in our analysis.


\begin{lemma} \label{nonlinear}
Let \(G\) be a group, let \(\pi\) be a set of primes, and assume that \(G\) has a nilpotent normal Hall \(\pi\)-subgroup \(N\). Denoting by \(H\) a \(\pi\)-complement of \(G\), assume also $\cent H N < H$.  If $\theta \in \irr N$ is nonlinear, then $\theta$ vanishes on some element of $N$ that is not conjugate to any element in $\cent N H$.
\end{lemma}

\begin{proof}
As we are assuming $\cent H N < H$, we clearly have also $\cent N H < N$.  It follows that the normal closure $\cent N H^N$ is strictly contained in $N$, since $N$ is nilpotent.  Also, again by the nilpotency of $N$, we can find a proper subgroup $X$ of $N$ and a linear character $\lambda \in \irr X$ such that $\lambda^G = \theta$.  Since $X < N$, we have $X^N < N$.  We know that $\theta$ vanishes on every element of $N$ that does not lie in any conjugate of $X$.  In particular, $\theta$ vanishes on every element of $N \setminus X^N$.  If we suppose that $\theta$ does not vanish on any element of $N$ that is not conjugate to $\cent N H$, then every element in $N \setminus X^N$ will be conjugate to $\cent N H$ and will lie in $\cent N H^N$.  This implies that $N = X^N \cup \cent N H^N$.  Since $X^N$ and $\cent N H^N$ are both proper subgroups of $N$, this is a contradiction.  Thus, $\theta$ must vanish on some element not conjugate to $\cent N H$.
\end{proof}

As already mentioned, I.M. Isaacs, G. Navarro, and T.R. Wolf introduced the study of vanishing elements in \cite{wolf}.  In that paper, they prove the following theorem.  As is customary, we denote by \(\fit G\) the Fitting subgroup of the group \(G\).

\begin{theorem}\cite[Theorem B]{wolf} \label{INW}
If $G$ is a supersolvable group, then the nonvanishing elements of $G$ all lie in \(\zent{\fit G}\). In particular, if $G$ is nilpotent, then the nonvanishing elements of $G$ are central.
\end{theorem}

We are now in a position to prove the main theorem in this section.

\begin{theorem} \label{nilpotent}
Let $G$ be a group, let $s$ be an integer, and let $\pi$ be the set of primes that divide $s$.  Then $\vcs G = \{ s \}$ and $G$ has a nilpotent Hall $\pi$-subgroup if and only if $G = NH$, where $N$ is a normal Hall $\pi$-subgroup and $H$ is an abelian $\pi$-complement, and one of the following occurs.
\begin{enumeratei}
\item The integer \(s\) is a prime power, $G = N \times H$, and ${\rm cs} (N) = \{ 1, s \}$.
\item $G/\zent G$ is a Frobenius group whose Frobenius kernel $N\zent G/\zent G$ has  order $s$, a Frobenius complement is $H\zent G/\zent G$, and $N$ contains no vanishing elements of $G$.
\end{enumeratei}
\end{theorem}

\begin{proof}
We first suppose that $\vcs G = \{ s \}$, and that \(G\) has a nilpotent Hall \(\pi\)-subgroup.  As already observed, by Lemma \ref{first}(a) we know that $G = NH$, where $N$ is a normal Hall $\pi$-subgroup and $H$ is an abelian $\pi$-complement of \(G\).  If $H = \cent H N$, then $G = H \times N$ and, since $N$ is nilpotent, Theorem~\ref{INW} yields that every element in $N\setminus\zent N$ is vanishing in \(G\).  It follows that ${\rm cs} (N) = \{1\}\cup \vcs G= \{1, s\}$.  We now apply the main result of \cite{ito}, getting that $s$ is a prime power and obtaining conclusion (a).

Assume then $\cent H N < H$. 
Lemma~2.6 of \cite{DPSS} yields that every element of $G \setminus N \cent H N$ is vanishing in \(G\), therefore, by our assumptions, these elements all lie in conjugacy classes of size $s$.  By Corollary~\ref{same size 1}(a) we see that $\cent N H$ is abelian, and we can also apply Lemma \ref{first}(b) to get that every vanishing element of $G$ contained in $N$ lies in some conjugate of $\cent N H$.  

Suppose that $N$ has a nonlinear $H$-invariant irreducible character \(\nu\).  By Lemma \ref{nonlinear}, there is an element $x \in N$ such that $\nu (x) = 0$ and $x$ is not conjugate in $N$ to any element of $\cent N H$.  Consider $\chi \in \irr {G}$ whose restriction to \(N\) has \(\nu\) as a constituent: it follows that $\chi_N = a \nu$ for some positive integer $a$, and so $\chi (x) = a \nu (x) = 0$.  Hence, $\chi$ vanishes on $x$ and we have a contradiction. Our conclusion so far is that $N$ has no nonlinear  $H$-invariant irreducible characters, and therefore the number of \(H\)-invariant irreducible characters of \(N\) is given by \(|\cent{N/N'} H|=|\cent N HN'/N'|\) (here we are using some well-known properties of coprime actions). The Glauberman-Isaacs correspondence, namely, Theorem~(13.1) of \cite{text}, yields now \(|\cent N HN'/N'|=|\cent N H|\), which forces $\cent N H \cap N'$ to be trivial. 

Our aim is now to prove that $N$ does not contain any vanishing element of $G$ and, for a proof by contradiction, we suppose that \(x\in N\) is a vanishing element of $G$: we claim first that $\zent N \le \cent N H$. In fact, by Lemma \ref{first}(b), $x$ lies in a conjugate of $\cent N H$ (by an element of \(N\)) and, replacing \(x\) by a suitable conjugate in \(N\), we may assume $x \in \cent N H$. Recalling that $\cent N H$ is abelian, we thus have $\cent N H \le \cent N x$. Let $h \in H \setminus \cent H N$; by Corollary~\ref{same size 1}(a) we have that $\cent G h = \cent N h H = \cent N H H \le \cent N xH = \cent G x$, but, since $h$ and $x$ are both vanishing elements of $G$, the indices of their centralizers in $G$ are both equal to \(s\).  As a consequence we get $\cent N H = \cent N x$ and, since $\zent N \le \cent N x$, we conclude that $\zent N \le \cent N H$, as claimed. Now, as $\cent N H < N$, this implies that $N$ is not abelian, whence $N' > 1$.  With $N$ being nilpotent, we get $\zent N \cap N' > 1$, and we deduce that $\cent N H \cap N' > 1$, a contradiction.  This proves the claim that $N$ does not contain any vanishing element of $G$.

Note that, if $\cent N H$ is normal in $G$, then, by Lemma \ref{hypothesis1}(a), $\cent N H = \zent G \cap N$.  Also, Corollary~\ref{same size 1}(b) yields that $G/\zent G$ is a Frobenius group whose Frobenius kernel is $N\zent G/\zent G$.  Since $H\zent G/\zent G$ is a Frobenius complement of \(G/\zent G\), for $h \in H \setminus \cent H N$ we have $h^g\zent G = h\zent G$ if and only if $g \in H\zent G\leq\cent GH$, hence $s = |h^G| = |(h\zent G)^{G/\zent G}| = |N\zent G/\zent G|$.

To prove conclusion (b) now, it suffices to show that $\cent N H$ is normal in $N$.  In fact, we show that $\cent N H$ is central in $N$.  Suppose that we can find an element \(x\) in $\cent N H\setminus\zent N$.  We know that every element of $N$ outside of $\zent N$ is a vanishing element of $N$.  Thus, there is a character $\nu \in \irr N$ so that $\nu (x) = 0$.  By Lemma \ref{hypothesis1}(c), $x$ is a vanishing element of $G$ and this contradicts the fact that $N$ does not contain any vanishing element of $G$, concluding our proof. 

\smallskip
Conversely, if \(G\) is as in (a) or (b), then clearly the normal Hall \(\pi\)-subgroup \(N\) of \(G\) is nilpotent. Moreover, if we are in case (a), then we have \(\cs G=\{1,s\}\) and \(\vcs G=\{s\}\) (because \(N\) is nonabelian, and so is \(G\)). On the other hand, suppose we are in case~(b), so, in particular, $N$ contains no vanishing elements of $G$.  By Lemma \ref{hypothesis1}(b) we know that $\cent H N = \zent G \cap H$, and  (using Clifford Theory) it can be checked that $N\zent G = N \times \cent H N$ contains no vanishing elements of $G$.  Therefore, the vanishing elements of \(G\) are (precisely the) elements of $G \setminus N\zent G$, and  
it is not difficult to see that, for an element \(g\) of this kind, the class size of $g\cent H N$ in $G/\cent H N$ is $s$.  Using Lemma \ref{hypothesis1}(b) again, we conclude that $\vcs G = \{ s \}$, as desired.
\end{proof}

As mentioned in the Introduction, there exist groups \(G\) as in conclusion (b) of the above theorem (and of Theorem~A) for which \(N\zent G/\zent G\) is nonabelian. Such a group is, for instance, the normalizer of a Sylow \(2\)-subgroup in the Suzuki group \({\rm{Suz}}(8)\), which is a Frobenius group of order \(2^6\cdot 7\) whose vanishing classes have all size \(2^6\). Moreover, considering now $G = {\rm SL} (2,3)$, we have \(G=NH\) where \(N\) is a normal Sylow \(2\)-subgroup and \(H\) a \(2\)-complement, and $G/\zent G$ is a Frobenius group whose Frobenius kernel \(N/\zent G\) has order $4$; but the set of vanishing class sizes of $G$ contains two numbers, $4$ and $6$, as in fact \(N\) does contain vanishing elements of~\(G\).


We can now handle the case where $s$ is a prime power, which is one part of Theorem~A.

\begin{corollary} \label{prime power}
Let $G$ be a group, let $p$ be a prime, and let $\alpha$ be a positive integer.  Then $\vcs G = \{ p^\alpha \}$ if and only if $G = PH$, where $P$ is a normal Sylow $p$-subgroup and $H$ is an abelian $p$-complement, and one of the following occurs.
\begin{enumeratei}
\item $G = P \times H$, and ${\rm cs} (P) = \{ 1, p^{\alpha} \}$.
\item $G/\zent G$ is a Frobenius group whose Frobenius kernel $P\zent G/\zent G$ has order $p^\alpha$, a Frobenius complement is $H\zent G/\zent G$, and $P$ contains no vanishing elements of $G$.
\end{enumeratei}
\end{corollary}

\begin{proof}
As a Sylow \(p\)-subgroup of \(G\) is nilpotent, this is an application of Theorem~\ref{nilpotent}.
\end{proof}

To close this section, we take time to point out an alternate description of groups as in case~(b) of Theorem \ref{nilpotent}. Basically, we see that the Hall \(\pi\)-subgroup of the center of these groups is a direct factor.

\begin{proposition}
Let $G$ be a group as in case~{\rm{(b)}} of Theorem~\(\ref{nilpotent}\).  Then $G = \cent N H \times M$, where $\cent N H$ lies in \(\zent G\), $M = [N,H]H$ is the normal closure of \(H\) in \(G\), and $M/\cent H N$ is a Frobenius group whose Frobenius kernel $[N,H] \cent H N/\cent H N$ contains no vanishing element of $M/\cent H N$. Furthermore, the unique vanishing class size of \(G\) is $|[N,H]|$.
\end{proposition}

\begin{proof}
We know that $G/\zent G$ is a Frobenius group whose kernel is $N\zent G/\zent G$ and  $H\zent G/\zent G$ is a Frobenius complement.  It follows that $\cent N H \le \zent G$, and so, $\cent N H = \zent G \cap N$ is abelian.   Arguing as in the proof of Theorem \ref{nilpotent}, we have $\cent N H \cap N' = 1$.  By coprimality, and taking into account that \(N/[N,H]\simeq \cent N H/\cent{[N,H]}H\) is abelian, we have $N/N' = \cent N HN'/N' \times [N,H]/N'$.  It follows that $\cent N H \cap [N,H] \le N'$, and so, $\cent N H \cap [N,H]  = \cent N H \cap N' \cap [N,H] = 1$.  This implies that $N = \cent N H \times [N,H]$.  Since $\cent N H$ centralizes $H$, we have $G = \cent N H \times [N,H]H = \cent N H \times M$.  Now, $M/\cent HN \simeq G/\zent G$ is a Frobenius group whose Frobenius kernel $[N,H]\cent HN/\cent HN\simeq[N,H]$ has order equal to the unique vanishing class size of \(G\).  Finally, note that any vanishing element of $M/\cent H N$ lying in $[N,H]\cent HN/\cent HN$ would give rise to a vanishing element of \(G\) lying in $N$, which does not exist by our hypothesis.\end{proof}

\section{When the unique vanishing class size is square-free}

In this section we will complete the proof of Theorem~A, focusing on the situation where $G$ is a group such that $\vcs G$ consists of a single element \(s\) \emph{which is a  square-free number}.   Groups where all of the vanishing class sizes are square-free were considered by J. Brough in \cite{brough}.  In that paper, the following theorem is proved.

\begin{theorem}\cite[Theorem B]{brough} \label{brough}
Let $G$ be a group and suppose that every vanishing conjugacy class size of \(G\) is square-free. Then $G$ is a supersolvable group.
\end{theorem}

In light of Theorem~\ref{brough}, the groups we study in this section are supersolvable. 

%
%

After two preliminary lemmas and a result from the literature, in Theorem~\ref{seven} we will start our analysis considering the case when \(\pi(s)\) is the whole \(\pi(G)\).

\begin{lemma}\label{lemma 1} 
Suppose the group $G$ is a semidirect product of $K \unlhd G$ and $L \leq G$. If $x$ is a vanishing element of $L$, then $x$ is a vanishing element of $G$ as well. In particular, if $L$ is nilpotent and $x$ lies in \(L\setminus\zent L\), then $x$ is vanishing in $G$.
\end{lemma}

\begin{proof}
Set $\overline{G} = G/K \simeq L$, and adopt the bar convention. Then $\overline{x}$ is vanishing in $\overline{G}$, and therefore \(x\) is vanishing in $G$.  By Theorem \ref{INW}, when $L$ is nilpotent all of the elements outside the center of $L$ are vanishing in $L$, and hence they are vanishing in $G$.
\end{proof}

\begin{lemma} \label{prime quotients}
Suppose that $G$ is a group such that $\vcs G$ consists of a unique element \(s\), let \(p\) be a prime divisor of \(s\) and let \(P\) be a Sylow \(p\)-subgroup of \(G\). Also, let $X$ be a normal subgroup of \(G\) with $\zent {\fit G} \le X \le \fit G$, such that every element in $G \setminus X$ is vanishing in \(G\). Then we have $\zent P \le \zent {\fit G}$, thus $G/X$ is a \(p'\)-group if \(P\) is abelian.   
Furthermore, $G/X$ is cyclic of order coprime with $s$ or all elements of $G/X$ have prime order.  
\end{lemma}

\begin{proof}
Let $x$ be an element of $\zent P$: since \(p\) does not divide \(|G:\cent G x|\), the conjugacy class of \(x\) in \(G\) has a size different from \(s\), and therefore \(x\) is not a vanishing element of \(G\). By our assumptions, \(x\) is then forced to lie in \(X\), thus in \(\fit G\). Clearly \(x\) is centralized by the Hall \(p'\)-subgroup of \(\fit G\) but, as \(x\in\zent P\), it is centralized by the Sylow \(p\)-subgroup of \(\fit G\) as well. In other words, \(x\) lies in \(\zent{\fit G}\) and our first claim follows. 


Now, our hypotheses imply that every conjugacy class of \(G\) lying in \(G\setminus X\) has size \(s\), therefore we are in a position to apply Theorem~\ref{isaacs}: if $G/X$ is cyclic, not of prime order, and \(r\) is a prime divisor of $|G/X|$, then the Sylow \(r\)-subgroups of \(G\) are abelian. In particular, by what proved in the paragraph above, \(r\) doesn't lie in \(\pi(s)\), and we conclude that the order of \(G/X\) is coprime with \(s\). In view of  Theorem~\ref{isaacs}, the only other possibility is that every element of $G/X$ has prime order.
\end{proof}

As we see in Lemma \ref{prime quotients}, groups where all the elements have prime order play a role in our work.  The following theorem classifies these groups.

\begin{theorem}\cite{cheng} \label{cheng}
Let $G$ be a group having all (nontrivial) elements of prime order. Then the following conclusions hold.
\begin{enumeratei}
\item $G$ is nilpotent if and only if $G$ is a $p$-group of exponent $p$.
\item $G$ is solvable and non-nilpotent if and only if $G$ is a Frobenius group with kernel $P \in {\rm Syl}_p(G)$, with $P$ a $p$-group of exponent $p$ and complement $Q \in {\rm Syl}_q(G)$, with $|Q|=q$. 
\item $G$ is nonsolvable if and only if $G$ is isomorphic to the alternating group $A_5$.
\end{enumeratei}
\end{theorem}

We now are ready to handle the case when $s$ is square-free and $\pi(s) = \pi(G)$.  

\begin{theorem} \label{seven}
Let \(G\) be a group, and let \(s\) be a square-free number such that \(\pi(s)=\pi(G)\). If $\vcs G = \{s\}$, then $s$ is a prime, $G$ is an $s$-group, and \(\cs G=\{1,s\}\). \end{theorem}

\begin{proof} 
Let us set \(F=\fit G\). As already observed, \(G\) is supersolvable  by Theorem~\ref{brough}, and Theorem~\ref{INW} yields that every element of \(G\setminus\zent F\) is vanishing in \(G\). Since \(\pi(G)\) consists of the prime divisors of \(s\), clearly \(|G/\zent F|\) (which cannot be \(1\)) is not coprime with \(s\), whence, in view of Lemma~\ref{prime quotients}, every element of \(G/\zent F\) has prime order; moreover, if \(p\) is a prime divisor of \(|G/\zent F|\) and \(P\) is a Sylow \(p\)-subgroup of \(G\), then we have \(\zent P\leq\zent F\) (in fact, \(\zent P\) lies in the Sylow \(p\)-subgroup of \(\zent F\), that we denote by \(\zent F_p\)) and \(P\) is nonabelian. Now, taking into account Theorem~\ref{cheng} and the fact that, \(G\) being supersolvable, \(G/\zent F\) is obviously not isomorphic to \(A_5\), we have to deal with two possible situations: either \(G/\zent F\) is a \(p\)-group of exponent \(p\), or it is a Frobenius group whose order is  divisible by exactly two primes.

\smallskip
Let us assume first that \(G/\zent F\) is a \(p\)-group. If \(G\) itself is a \(p\)-group, then the result follows at once just recalling that \(\cs G=\{1\}\cup \vcs G\) by Theorem~\ref{INW}; therefore, we can assume that there exists a nonempty set \(\{q_1,\ldots,q_k\}\) of primes different from \(p\) such that \(\pi(G)=\{p,q_1,\ldots,q_k\}\). Observe that, if \(Q_i\) is a Sylow \(q_i\)-subgroup of \(G\), then we have \(\zent F=Q_1\times\cdots\times Q_k\times\zent F_p\), and \(G\) is a semidirect product \((Q_1\times\cdots\times Q_k)\rtimes P\). By Lemma~\ref{lemma 1}, every element in \(P\setminus\zent P\) is vanishing in \(G\); since the elements in \(\zent F_p\) lie in conjugacy classes whose sizes are powers of \(p\), they are nonvanishing in \(G\), and therefore we get \(\zent F_p\leq \zent P\) (so equality holds, by what was observed in the previous paragraph). 

Now, consider the conjugation action of \(P\) on \(Q_1\). Observe that no element \(x\) of \(P\setminus \zent P\) centralizes \(Q_1\), as otherwise the conjugacy class of \(x\) in \(G\) would have a size not divisible by \(q_1\), contradicting the fact that \(x\) is a vanishing element of \(G\); as a consequence, \(\overline{P}=P/\zent P\) acts faithfully on \(Q_1\). Furthermore, for every nontrivial \(\overline x\in\overline P\), we see that \(|Q_1:\cent {Q_1} {\overline x}|\) is a divisor (different from \(1\)) of \(|G:\cent G x|\), therefore it is \(q_1\). Since \(Q_1=\cent {Q_1}{\overline x}\times[Q_1,\langle \overline x\rangle]\) by coprimality, we get \(|[Q_1,\langle \overline x\rangle]|=q_1\). We can now appeal to Lemma~2.4 of \cite{babo} to conclude that \(P/\zent P\) is cyclic, which is a contradiction because it would imply that \(P\) is abelian.

\smallskip	
Finally, suppose $G/\zent F$ is a Frobenius group whose order is divisible by exactly two primes \(p\) and \(r\), and whose Frobenius kernel \(K/\zent F\) has index \(r\) in $G/\zent F$ (observe that $F \le K$).  Let $R$ be a Sylow $r$-subgroup of $G$.   Now, $R\zent F/\zent F$ is an abelian Frobenius complement, and thus, $R\zent F/\zent F$ is cyclic.  If we can show that $\zent R = \zent F \cap R$, then we have that $R/\zent R$ is cyclic which implies that $R$ is abelian, and hence, $R = \zent R$ which is a contradiction.

Thus, it suffices to show $\zent R = \zent F \cap R$.  In the first paragraph of the proof, we saw that $\zent R \le \zent F$, so $\zent R \le \zent F \cap R$.  We now work to show $\zent F \cap R \le \zent R$. Labeling the primes in $\pi \setminus \{ p, r \}$ as $q_1, \dots, q_k$,  we can write $F =  Q_1 \times \dots \times Q_k\times F_r\times F_p$ as the direct product of its Sylow subgroups.  Let $P$ be a Sylow $p$-subgroup of $G$, and observe that $K = (Q_1 \times \dots \times Q_k\times F_r) \rtimes P$.  Since $K/\zent F$ is a $p$-group, we have $F_r \le \zent F$ and, in particular, $F_r$ centralizes $F_p$.  Let $x \in P \setminus F_p$.  Using the fact that $G/\zent F$ is a Frobenius group, we have $\cent {G/\zent F}{x\zent F} \le K/\zent F$, and so, $\cent G x \le K$.  On the other hand, as $s$ is square-free, we have $|G:\cent G x|_r = r$.  Because $|G:K| = r$, this implies $F_r \le \cent G x$.  Since $x$ was arbitrary in $P \setminus F_p$ and $F_r$ centralizes $F_p$, this implies that $F_r$ centralizes $P$.  Hence, as clearly \(F_r\) also centralizes the \(Q_i\), we get $K = F_r \times ((Q_1 \times \dots \times Q_k) \rtimes P)$.  Notice that $L = (Q_1 \times \dots \times Q_k) \rtimes P$ will now be normal in $G$, and so we may apply Lemma \ref{lemma 1} to see that every noncentral element of an \(r\)-Sylow $R$ of \(G\) is a vanishing element of $G$.  Since the elements of \(F_r\) are nonvanishing in \(G\) (as their conjugacy class sizes are coprime to \(p\)), this implies $F_r \leq \zent R$.  Since $\zent F \cap R \le F_r$, this yields the result.
\end{proof}

Next, we consider the case where $\pi (s)$ is properly contained in $\pi (G/\zent G)$.  In this case, we make use of the following theorem.

\begin{theorem}\cite[Theorem D]{wolf}\label{ThmD}
If $G$ is a solvable group and $x \in G$ is nonvanishing, then $x \fit G$ has $2$-power order in $G/\fit G$.  In particular, if $x$ has odd order, then $x \in \fit G$.
\end{theorem}

As we will see, in this situation we do not need to assume that $s$ is square-free to obtain the desired conclusion; we only need to assume that $G$ is supersolvable.  Also, the argument when $G$ has odd order is very similar in this case, so we have included it.

\begin{theorem} \label{supersolvable 1}
Let \(G\) be a group, let \(s\) be an integer, and let \(\pi\) be the set of primes that divide \(s\). Assume that \(\pi\) is strictly contained in \(\pi(G/\zent G)\), and that $G$ is supersolvable or $G/\fit G$ has odd order. If ${\rm vcs}(G) = \{s\}$, then \(G\) has a nilpotent Hall \(\pi\)-subgroup.
\end{theorem}

\begin{proof}
By Lemma \ref{first}(a) we know that $G = NH$, where $N$ is a normal Hall $\pi$-subgroup of $G$ and $H$ is an abelian $\pi$-complement of $G$.  Also, by Lemma \ref{hypothesis1}(b), we have $\cent H N = \zent G \cap H$.  
The fact that $\pi$ is strictly contained in $\pi (G/\zent G)$ implies that $H\zent G/\zent G$ is nontrivial, so $\cent H N = \zent G\cap H < H$. Now, the factor group \(G/\cent H N\) (whose center is \(\zent G/\cent H N\)) satisfies the hypotheses of our statement by Lemma~\ref{hypothesis1}(b), and, if we prove the claim for \(G/\cent H N\), then we clearly get the desired conclusion for \(G\) as well. As a consequence, there is no loss of generality in assuming $\cent H N=1$, thus the Fitting subgroup \(F\) of $G$ lies in $N$.  Note that, when $G$ is supersolvable, all elements in $G\setminus\zent F$ are vanishing in $G$ by Theorem~\ref{INW}; but by Theorem \ref{ThmD}, all elements in $G\setminus F$ are vanishing in $G$ if $G/F$ has odd order. Therefore, under our assumptions, in any case we have that all of the elements of \(G\) lying outside \(F\) are vanishing in \(G\).




We now appeal to Lemma \ref{prime quotients}:  if $G/F$ is cyclic and not of prime order, then $|G/F|$ is not divisible by any prime in $\pi$.  This implies $F = N$, and the conclusion follows. Thus, again by Lemma \ref{prime quotients}, we may assume that all nonidentity elements of $G/F$ are of prime order and we can use the classification given in Theorem \ref{cheng}, taking into account that $G/F$ is obviously not isomorphic to $A_5$.

Suppose $G/F$ is a $p$-group of exponent $p$ for some prime $p$. Since $G/N$ and $N/F$ have coprime orders, we must have either $G = N$  or $N = F$. But if $G = N$, then we have a contradiction to the fact that $\pi$ is strictly contained in $\pi(G/\zent G)$.  Thus, we must have $F=N$, and again we are done. 

Finally, we consider the case in Theorem \ref{cheng} when $G/F$ is a Frobenius group whose Frobenius kernel is a $p$-group with exponent $p$ for some prime $p$ and has index $r$ for a different prime $r$.  We know that every normal subgroup of $G/F$ either contains or is contained in the Frobenius kernel of $G/F$.  Since the Frobenius kernel has index $r$, and $N$ is proper in $G$, we see that $N/F$ is contained in the Frobenius kernel.  Now, if $N/F$ were proper in the Frobenius kernel, then $H\simeq G/N$ would be nonabelian, which is not the case. Thus $N/F$ is in fact the Frobenius kernel of $G/F$, and it follows that $r$ is not in $\pi$.
If, arguing by contradiction, we have $F < N$, then we can choose $n \in N \setminus F$.  Observe that $\cent G nF/F \le \cent {G/F} {nF} \le N/F$.  This implies that $r$ divides $|G:\cent G n|$.  Since $n \in N \setminus F$, we see that $n$ is a vanishing element of $G$, and so, $s = |G:\cent G n|$.  This contradicts the fact that, as observed above, $r$ does not divide $s$, and the proof is complete. 
\end{proof}



The following corollary will complete the proof of Theorem~A.

\begin{corollary} 
Let \(G\) be a group, and let \(s\) be an integer that is square-free. Then, denoting by \(\pi\) the set of prime divisors of \(s\), we have \(\vcs G=\{s\}\) if and only if \(G=NH\), where \(N\) is a normal Hall \(\pi\)-subgroup and \(H\) an abelian \(\pi\)-complement of \(G\), and one of the following occurs.
\begin{enumeratei}
\item The integer \(s\) is a prime, $G = N \times H$, and \(\cs N=\{1,s\}\).
\item $G/\zent G$ is a Frobenius group whose Frobenius kernel \(N\zent G/\zent G\) has order $s$, a Frobenius complement is \(H\zent G/H\), and \(N\) contains no vanishing elements of \(G\).
\end{enumeratei}
\end{corollary}

\begin{proof} 
Assuming \(\vcs G=\{s\}\), we have that $G$ is supersolvable by Theorem~\ref{brough}.  By Lemma~\ref{first}(a), we know that $G = NH$ where $N$ is the normal Hall $\pi$-subgroup of $G$ and $H$ is an abelian $\pi$-complement.  If $\cent H N = H$, then $G = N \times H$, and $\vcs G = \vcs N$, so we get conclusion~(a) by Theorem~\ref{seven}. On the other hand, if $\cent N H < H$, then we use Lemma \ref{hypothesis1}(b) to see that $\cent H N = \zent G \cap H$ and conclude that $\pi$ is properly contained in $\pi (G/Z (G))$. Now, by Theorem~\ref{supersolvable 1} we have that \(N\) is nilpotent, and conclusion (b) follows by Theorem~\ref{nilpotent}. As for the converse statement, this is the ``if part" of Theorem~\ref{nilpotent}.
\end{proof}

\section{Theorem B}

The aim of this section is to prove Theorem~B, and also to conclude our description (started in Theorem~\ref{supersolvable 1}) of groups having a unique vanishing conjugacy class size and a normal Sylow \(2\)-subgroup. 

\begin{theorem} \label{pi(s)=pi(G)}
Let $G$ be a group and \(s\) a positive integer such that \(\pi(s)=\pi(G)\). Assume that  \(G\) is supersolvable or \(|G:\fit G|\) is odd. If $\vcs G = \{ s \}$, then one of the following occurs.
\begin{enumeratei}
\item The integer \(s\) is a power of a suitable prime \(p\), $G$ is a $p$-group and $\cs G = \{ 1, s \}$.
\item For a suitable prime \(p\) we have $G = N P$, where $N$ is a nontrivial, nilpotent, normal $p$-complement of \(G\), and $P$ is a Sylow $p$-subgroup of \(G\) such that $\cs P = \{ 1, s_p \}$, $\zent P = \cent P N = \zent G \cap P$ and \(P/\zent P\) has exponent \(p\). Also, $|N:\cent N x| = s_{p'}$ and $\cent N x = \cent N {\cent P x}$ is abelian, for every element $x \in P \setminus \zent P$. Finally, $\cent N P =N\cap  \zent G$, every element in \(\zent N\) is nonvanishing in \(G\) and, if \(G\) is supersolvable, then \(N\) is abelian and \(P/\zent P\) is an elementary abelian \(p\)-group.
\item For suitable distinct primes \(p\) and \(q\) we have $G = NH$, where $N$ is a nilpotent normal $\{ p, q\}$-complement of \(G\), and $H$ is a non-supersolvable Hall $\{p,q\}$-subgroup of \(G\) such that: $\vcs {H}$ consists only of $s_{\{p,q\}}$, which is divisible by \(q^2\), and $H/\fit H$ is a Frobenius group whose Frobenius kernel is a $p$-group of exponent $p$ having index $q$.  In addition, if $P$ and $Q$ are a Sylow $p$- and a Sylow $q$-subgroup of $H$ respectively, then $\fit H = \fit G \cap H$ contains $\zent P$ and $\zent Q$. Also, $|N:\cent N x| = s_{\{p,q\}'}$ and $\cent N x = \cent N{\cent H x}$ is abelian, for every element $x \in H \setminus \fit H$.  Finally, $\cent N H = N\cap \zent G$, and every element in $N\oh p G$ is nonvanishing in \(G\). 
\end{enumeratei}
Furthermore, if $G$ satisfies {\rm(a)}, then $\vcs G = \{ s \}$. If $G$ satisfies {\rm(b)} and $N$ contains no vanishing elements of $G$, then $\vcs G = \{ s \}$. Finally, if $G$ satisfies {\rm(c)} with $|G/\fit G|$ odd and \(\fit G\) containing no vanishing elements of \(G\), then $\vcs G = \{ s \}$.
\end{theorem}

\begin{proof}
Set \(F=\fit G\): in view of Theorem~\ref{INW} and of Theorem \ref{ThmD}, our assumptions imply that every element of \(G\setminus F\) is a vanishing element of \(G\). If \(|G/F|\) is coprime with \(\pi(s)=\pi(G)\), then \(|G/F|=1\); in other words, \(G\) is nilpotent and we easily get conclusion (a). On the other hand, if this does not happen, then  Lemma~\ref{prime quotients} yields that every element of \(G/F\) has prime order, and again we are in a position to use the classification provided by Theorem~\ref{cheng} (of course with no need to consider the case \(G/F\simeq A_5\)).

\smallskip
So, let us assume first that \(G/F\) is a \(p\)-group of exponent \(p\), where \(p\) is a suitable prime. Our aim is to show that, under this hypothesis, conclusion (b) holds. 

In this case clearly \(G\) has a nontrivial, nilpotent, normal \(p\)-complement \(N\) and, if \(P\) is a Sylow \(p\)-subgroup of \(G\), then we have \(G=NP\). Observe that, by Lemma~\ref{prime quotients}, \(\zent P\) lies in \(\zent F\), thus it lies in the maximal normal \(p\)-subgroup \(\oh p G\) of \(G\), which in turn lies in \(\cent P N\). However, every element \(x\) in \(\cent P N\) has a conjugacy class whose size is coprime with  each of the prime factors of \(|N|\), thus \(x\) cannot be vanishing in \(G\) (note that, for the same reason, every element of \(\zent N\) is also nonvanishing in \(G\)): as a consequence, taking into account that every element of \(P\setminus\zent P\) is vanishing in \(G\) by Lemma~\ref{lemma 1}, we see that the inclusion \(\cent P N\leq \zent P\) holds as well. Now we have \(\zent P=\oh p G=\cent P N\), which immediately implies \(\cent P N=\zent G\cap P\). We also see that \(F=N\times\zent P\) and \(P/\zent P\simeq G/F\) has exponent \(p\). Finally, note that \(N\cent P N=F\), and therefore we can appeal to Theorem~\ref{same size} obtaining the following properties: for every \(x\in P\setminus\zent P\), we have $|P:\cent P x| = s_p$ (i.e., $\cs P = \{ 1, s_p \}$), $|N:\cent N x| = s_{p'}$, and $\cent N x = \cent N {\cent P x}$ is abelian. 

Next, observe that the elements of \(\cent N P\) are nonvanishing in \(G\), because their conjugacy class sizes are not divisible by \(p\); an application of Lemma~\ref{hypothesis1}(c) yields now that those elements are nonvanishing also in the nilpotent group \(N\), so they lie in \(\zent N\). We conclude that \(\cent N P\leq \zent N\), whence \(\cent N P=N\cap\zent G\). 

To conclude our analysis that leads to conclusion (b), it remains to prove that \(N\) is abelian if \(G\) is supersolvable. Under this additional assumption, we know (Theorem~\ref{INW}) that every element of \(G\setminus\zent F\) is vanishing in \(G\), thus, as in the first paragraph of this proof, we apply Theorem~\ref{cheng}: either \(G/\zent F\) is a \(p\)-group or it is a Frobenius group whose kernel has prime index in \(G/\zent F\). In the former case \(N\) is forced to coincide with \(\zent N\), as \(N/\zent N\simeq F/\zent F\) has order coprime to \(p\), and we get the desired conclusion. In the latter case it is easily seen that the (nilpotent) Frobenius kernel of \(G/\zent F\) must coincide with \(F/\zent F\); but then, \(P/\zent P\simeq G/F\) would have prime order and so \(P\) would be abelian, which is a contradiction.

Note also that, when \(G\) is supersolvable, we have \(G'\leq F\). So \(P/\zent P\) (being isomorphic to \(G/F\)) is abelian and, since it has exponent \(p\), it is an elementary abelian \(p\)-group.

\smallskip

The other possibility given by Theorem \ref{cheng} is that $G/F$ is a Frobenius group whose kernel \(K/F\) is a \(p\)-group of exponent \(p\) and of index \(q\) in \(G/F\), where \(p\) and \(q\) are suitable primes. In this case we will reach conclusion (c). To begin with, \(G\) has obviously a nilpotent normal \(\{p,q\}\)-complement \(N\), so that \(G=NH\) where \(H\) is a Hall \(\{p,q\}\)-subgroup of \(G\). Denoting by \(P\) and \(Q\) a Sylow \(p\)- and a Sylow \(q\)- subgroup of \(H\) respectively, we get \(H=PQ\). 

Observe that, if \(N\) is nontrivial, then every element in \(\cent H N\) is nonvanishing in \(G\) because its conjugacy class size is coprime with any prime divisor of \(|N|\). As a consequence, we have \(\cent H N\leq F\) in this case, whence \(\cent H N=H\cap F\) and  \(N\cent H N=F\). Now Theorem~\ref{same size} yields the following properties: $\cent N x = \cent N {\cent H x}$ is abelian, $|N:\cent N x| = s_{\{p,q\}'}$, and $|H:\cent H x| = s_{\{p,q\}}$ for all elements $x \in H \setminus (H \cap F)$. Recall also that, if \(x\) is a vanishing element of \(H\), then Lemma~\ref{lemma 1} yields that \(x\) is also a vanishing element of \(G\), and therefore it lies in \(H\setminus (H\cap F)\); we conclude that \(\vcs H=\{s_{\{p,q\}}\}\). On the other hand,  if $N = 1$, then \(H=G\), $s = s_{\{p,q\}}$, and we trivially have the same properties.



Notice that if $g \in P \setminus \oh p G$ then, since $G/F$ is a Frobenius group with kernel \(K/F\) and \(gF\) is a nontrivial element of \(K/F\), we have $\cent G g F/F \le \cent {G/F} {gF} \le K/F$.  It follows that $\cent G g \le K$, so \(\cent K g=\cent G g\). We claim now that \(s\) is divisible by \(q^2\) and, arguing by contradiction, we assume $s_q = q$. As our element \(g\) doesn't lie in \(F\), we have $|G:\cent G g|_q = s_q = q$; since $q = |G:K|$ divides $|G:\cent G g|$, it follows that $\cent G g$ must contain \(\oh q G\), which is the Sylow $q$-subgroup of $K$. We have shown that $\oh q G \le \cent G g$ for every element $g \in P \setminus\oh p G$, but then clearly \(\oh q G\) centralizes in fact \(P\), and therefore every element of \(\oh q G\) is nonvanishing in \(G\). Also, $NP$ is now a normal \(q\)-complement of $K$, hence \(NP\) is normal in \(G\) and, by Lemma~\ref{lemma 1}, every element in $Q \setminus \zent Q$ is vanishing in \(G\). The conclusion so far is $\oh q G \le \zent Q$, but then \(|Q/\zent Q|\) is a divisor of \(|Q/\oh q G|=q\), which is a contradiction because it implies that \(Q\) is abelian. Our claim that \(s\) is divisible by \(q^2\) is then established. 

The property proved in the paragraph above implies that a vanishing element of \(G\) cannot centralize \(\oh q G\) (note that, by this reason, every element in \(N\oh p G\) is nonvanishing in \(G\)) and so, in particular, every element in \(P\setminus\oh p G\) doesn't lie in $\oh p H$. It follows that $\oh p H =\oh p G$. Note that we also have \(\oh q H=\oh q G\), because otherwise (recalling that \(|Q:\oh q G|=q\)) we would have \(\oh q H= Q\), and so \(Q\) would be normalized by \(P\); this clearly cannot happen, as \(G/F=PF/F\rtimes QF/F\) is a Frobenius group. Now, we get \(F\cap H=\oh p G\times\oh q G=\oh p H\times\oh q H=\fit H\), and the desired conclusion about the factor group \(H/\fit H\) being a Frobenius group follows.  Observe that, \(H/\fit H\) being nonabelian, we have that \(H\) is not supersolvable. Note also that no element of \(\zent P\) can be vanishing in \(G\), because its conjugacy class size is not divisible by \(p\), so \(\zent P\) lies in \(F\cap H\), and a similar argument shows that \(\zent Q\leq F\cap H\) as well. 


Finally, observe that every element of \(N\setminus\zent N\) is a vanishing element of \(N\), and if \(\cent N H\setminus\zent N\) is nonempty, then every element of that set is in fact vanishing in \(G\) by Lemma~\ref{hypothesis1}(c). But we have seen that no element of \(N\) can be vanishing in \(G\), so \(\cent N H\setminus\zent G\) is actually empty, and also the remaining claim that \(\cent N H=N\cap\zent G\) easily follows.


\smallskip
We move now to the converse statement. It is immediate to see that, if $G$ satisfies conclusion (a), then $\vcs G = \{ s \}$.  

Also, if $G$ satisfies conclusion (b), then it satisfies the hypotheses of Theorem~\ref{same size} (``if part") and so every element in \(G\setminus N\cent P N=G\setminus F\) will lie in conjugacy classes of the same size (namely, of size \(|N:\cent N x|\cdot |P:\cent P x|\), for any \(x\in P\setminus\zent P\)). If every element of \(N\) is nonvanishing in \(G\), which certainly happens if \(N\) is abelian, then by Clifford Theory (and taking into account that \(\zent P\leq\zent G\)) it is not difficult to see that in fact every element in \(F=N\times\zent P\) is nonvanishing in \(G\). As a consequence, \(\vcs G\) will consist of a single element. 

To close with, suppose that $G$ satisfies conclusion (c) and \(G/F\) has odd order. Then also \(H/\fit H\simeq G/F\) has odd order and, since \(H\) is solvable, every element of \(H\setminus\fit H\) is vanishing in \(H\) by Theorem \ref{ThmD}. Therefore, by our assumptions, we have $|H:\cent H x| = s_{\{p,q\}}$ for all $x \in H \setminus \fit H$. Observe that, \(\fit H\) being contained in \(\fit G\), we have \(\fit H\leq\cent H N\); but if the inclusion is proper, then we can take \(x\in\cent H N\setminus\fit H\), getting that \(|N:\cent N x|=1\) for this particular \(x\), hence, by assumption, for every \(x\) in \(H\setminus \fit H\). As this is clearly impossible, we conclude that \(\cent H N=\fit H\), and so \(N\cent H N=F\). Now we are in a position to apply the ``if part" of Theorem~\ref{same size}, getting that all the elements of \(G\setminus F\) lie in conjugacy classes of size \(s_{\{p,q\}}\cdot s_{\{p,q\}'}=s\).  The assumption that $F$ contains no vanishing elements of $G$ yields the desired conclusion.
\end{proof}

We can now prove a slightly extended form of Theorem~B which is an ``if and only if" theorem.

\begin{theorem}\label{thmB}
Let $G$ be a group, \(s\) a positive integer, and set \(\pi=\pi(s)\). If \(G\) is supersolvable, $\vcs G = \{ s \}$ and \(\pi=\pi (G/\zent G)\), then, up to an abelian direct factor of \(\pi'\)-order, one of the following occurs.
\begin{enumeratei}
\item The integer \(s\) is power of a suitable prime \(p\), \(G\) is a \(p\)-group and $\cs G = \{ 1, s \}$.
\item For a suitable prime \(p\) we have $G = N P$, where $N$ is a nontrivial, abelian, normal $p$-complement of \(G\), and $P$ is a Sylow $p$-subgroup of \(G\) such that: $\cs P = \{ 1, s_p \}$, $\zent P = \cent P N = \zent G \cap P$ and \(P/\zent P\) is an elementary abelian \(p\)-group. Also, $|N:\cent N x| = s_{p'}$ and $\cent N x = \cent N {\cent P x}$ for every element $x \in P \setminus \zent P$. Finally, $\cent N P =N\cap  \zent G$ and \(N\) contains no vanishing elements of \(G\).
 \end{enumeratei}
 Conversely, if $G$ is as in {\rm(a)} or  {\rm(b)}, then $\vcs G=\{s\}$ and \(\pi=\pi(G/\zent G)\).
 \end{theorem}
 
\begin{proof}
Let \(\vcs G=\{s\}\). By Lemma \ref{first}(a) we know that $G = G_0H$, where $G_0$ is a normal Hall $\pi$-subgroup of $G$ and $H$ is an abelian $\pi$-complement of $G$; note that the prime divisors of \(H\) do not belong to \(\pi(G/\zent G)\), which easily yields \(H\leq\zent G\) and thus 
\(G=G_0\times H\). As \(\pi(G_0)=\pi\), the ``if part" of  the result follows by Theorem~\ref{pi(s)=pi(G)}.

The converse statement is just an application of Theorem~\ref{pi(s)=pi(G)}, and the proof is complete.
\end{proof}

Note in the case of (b) in the converse of Theorem \ref{thmB} that the group $G$ is not necessarily supersolvable.

\section{Some open problems, examples and concluding remarks}

Let us consider the general question of describing all the groups having a single vanishing conjugacy class size. As a first step, one may ask whether a group of this kind should be necessarily solvable; we leave this as an open problem, but we can present one small contribution in this direction.

\begin{proposition}\label{abelian}
Let \(G\) be a group. If $\vcs G = \{ s \}$, then all minimal normal subgroups of $G$ are abelian.  
\end{proposition}

\begin{proof}
Setting \(\pi=\pi(s)\), we apply Lemma~\ref{first}(a) to decompose \(G\) as \(NH\), where \(N\) is a normal Hall \(\pi\)-subgroup and \(H\) an abelian \(\pi\)-complement of \(G\).  Suppose, for a proof by contradiction, that $M$ is a nonabelian minimal normal subgroup of $G$.  Hence, there exists a prime $p\not\in\{2,3\}$ that divides $|M|$.  Since $G/N$ is abelian we get $M \le N$, and this implies $p \in \pi $ (i.e., $p$ divides $s$).

Let $P$ be a Sylow $p$-subgroup of $G$.  We know that $P \cap M$ is a nontrivial normal subgroup of $P$, thus $\zent P \cap M=\zent P \cap (P \cap M) > 1$.

Suppose $x \in \zent P$. As $P \le \cent G x$, we see that $p$ does not divide $|G:\cent G x|$, hence the conjugacy class of \(x\) in \(G\) has a size different from \(s\) and $x$ is a nonvanishing element of $G$.  Since $p$ is relatively prime to $6$, by Theorem A in \cite{DNPST} we get $x \in \fit G$.  Now, we have shown that $\zent P \le \fit G$ and $M \cap \zent P > 1.$  This implies $M \leq \fit G$, against the fact that $M$ is a nonabelian minimal normal subgoup of \(G\). 
\end{proof}

As a consequence of the above proposition, if the group $G$ has only one vanishing class size then $G$ is not an almost-simple group.  

\begin{rem} \label{generalization} As regards the classification of the \emph{solvable} groups \(G\) such that \(|\vcs G|=1\), giving a closer look at Theorem~\ref{supersolvable 1} and Theorem~\ref{pi(s)=pi(G)}, we see that the assumption of \(G\) being supersolvable or having a normal Sylow \(2\)-subgroup is really used only to guarantee that \emph{every element of \(G\setminus \fit G\) is vanishing in \(G\)}.
A long-standing open problem, going back to the paper~\cite{wolf}, is to determine whether the above property holds in general if \(G\) is solvable. If the answer to this conjecture would turn out to be positive, then our results in this paper would actually provide a characterization of solvable groups having a single vanishing conjugacy class size.
\end{rem}

To sum up, we strongly believe that the groups having a unique vanishing conjugacy class size are only those that are described in the conclusions of Theorem~\ref{nilpotent} and of Theorem~\ref{supersolvable 1} (although, as explained after the following examples, we are not sure that all of them actually occur). We leave this conjecture for future research.

\medskip

Next, we discuss some examples showing that groups as in conclusion (b) of Theorem~\ref{thmB} do exist.

\begin{example} For our first example, take $P$ to be an extra-special group of order $8$.  (I.e., $P$ can be either dihedral or quaternion.)  We take $N$ to be an elementary abelian group of order $r^3$ for $r = 3, 5, 7, 11$, so, we write \(N\) as  $C_1 \times C_2 \times C_3$, where each $C_i$ is cyclic of order $r$.  We then define the action of $P$ on $N$ by letting $M_1$, $M_2$, and $M_3$ be the three distinct maximal subgroups of $P$, and we have $P$ act on $C_i$ by taking $M_i$ in the kernel of the action.  It is easy to see that the resulting semi-direct products $G$ are supersolvable groups.  In Magma or GAP, when $r = 3$ this yields the groups SmallGroup(216,132) and SmallGroup(216,133), and computing the character tables of these groups in Magma, we see that they have $\vcs {G} = \{ 18 \}$.  When $r =5$, this yields SmallGroup(1000,142) and SmallGroup(1000,143), and computing the character tables of these groups in Magma, we see that they have $\vcs {G} = \{ 50\}$.  When $r = 7$ or $11$, the resulting groups are not in the SmallGroup database in Magma or GAP, but we can have Magma compute the character tables, and we see that $\vcs {G} = \{ 98 \}$ when $r  = 7$ and $\vcs {G} = \{ 242 \}$ when $r = 11$.  We were also able to compute the group where each of the $C_i$ are cyclic of order $9$ and $G/M_i$ acts fixed-point freely on $C_i$.  For this $G$, we have $\vcs G = \{ 162\}$.

We also present an example of odd order.  Take $P$ to be an extraspecial group of order $3^3$ and $N$ to be an elementary abelian $7$-group of order $7^4$.  Thus, we can write $N = C_1 \times C_2 \times C_3 \times C_4$, where each $C_i$ is cyclic of order $7$.  Let $M_1, M_2, M_3,$ and $M_4$ be the distinct maximal subgroups of $P$.  We can define an action of $P$ on $N$ by having $P$ act on $C_i$ with $M_i$ in its kernel for each $i = 1,2,3,4$.  Using Magma, we determine that the resulting semi-direct product $G$ has $\vcs {G} = \{ 1029 \}$. 

Modifying the last example, we can obtain an example where $G$ is not supersolvable.  Again, take $P$ to be an extraspecial group of order $3^3$.  Take $N = K_1\times K_2 \times K_3 \times K_4$ where each $K_i$ is a non-cyclic group of order $4$.  Let $M_1$, $M_2$, $M_3$, $M_4$ be the four distinct maximal subgroups of $P$.  There is an action of $P$ on $N$ by having $P$ act on $K_i$ so that $M_i$ is the kernel of the action for $i = 1, 2, 3, 4$.  Then take $G$ to be the resulting semi-direct product of $P$ acting on $N$.  Using Magma, we determine that $\vcs {G} = \{ 192 \}$. 
\end{example}

On the other hand, we do not have any example of a group as in (b) of Theorem~\ref{pi(s)=pi(G)} where \(N\) is nonabelian. We believe it is likely that such a group does exist, but the ones we can construct are too large for Magma to determine the vanishing classes, and at this time, we do not see another method to determine the vanishing classes of these groups.  We also don't know of any example of a group as in (c) of Theorem~\ref{pi(s)=pi(G)}. 

\bigskip

\subsection*{Acknowledgment}

Much of the research of this paper was done in January 2020 during a visit to  Universit\`a degli Studi di Milano by the second and third authors.  The second and third authors would like to thank Universit\`a degli Studi di Milano for its hospitality.

\end{document}